\documentclass[12pt]{article}
\input epsf.tex
\usepackage{graphicx}
\usepackage{amsmath,amsthm,amsfonts,amscd,amssymb,comment,eucal,latexsym,mathrsfs}
\usepackage{stmaryrd}
\usepackage[all]{xy}

\usepackage{epsfig}

\usepackage[all]{xy}
\xyoption{poly}
\usepackage{fancyhdr}
\usepackage{wrapfig}
\usepackage{epsfig}

\theoremstyle{plain}
\newtheorem{thm}{Theorem}[section]
\newtheorem{prop}[thm]{Proposition}
\newtheorem{lem}[thm]{Lemma}

\theoremstyle{definition}
\newtheorem{defn}{Definition}
\theoremstyle{remark}
\newtheorem{remark}{Remark}

\advance \topmargin by -\headheight
\advance \topmargin by -\headsep
\evensidemargin \oddsidemargin
     \def\F{{\mathbb{F}}}              \def\TT{{\mathbb{T}}}      \def\Z{{\mathbb{Z}}}

    \def\cE{{\mathcal{E}}} \def\cF{{\mathcal{F}}}            \def\cR{{\mathcal{R}}}  \def\cT{{\mathcal{T}}}      

   \def\hD{{\widehat{D}}} \def\hE{{\widehat{E}}}               \def\hT{{\widehat{T}}}      
     \def\hcF{{\widehat{\mathcal{F}}}}

  \def\homega{{\widehat{\omega}}}

                   \def\tT{{\tilde{T}}}      

\newcommand{\G}{\Gamma}
\newcommand{\Ga}{\Gamma}
\newcommand{\La}{\Lambda}

\renewcommand\a{\alpha}
\renewcommand\b{\beta}

\renewcommand\k{\kappa}

\newcommand\s{\sigma}

\newcommand\Aut{\operatorname{Aut}}

\newcommand\Cyl{\operatorname{Cyl}}

\newcommand\past{\operatorname{past}}

\newcommand\Prob{\operatorname{Prob}}

\def\cc{{\curvearrowright}}
\newcommand{\resto}{\upharpoonright}

\begin{document}
\title{All properly ergodic Markov chains over a free group are orbit equivalent}
\author{Lewis Bowen\footnote{supported in part by NSF grant DMS-1500389, NSF CAREER Award DMS-0954606} \\ University of Texas at Austin}
%\author{University of Hawaii}
%\author[Lewis Bowen]{Lewis Bowen$\dagger$}
%\address{Department of Mathematics\\
%University of Hawai'i--Manoa\\
%} %one \address command per author
%\email{lpbowen@math.hawaii.edu}
%%\thanks{$\dagger$ Supported in part by NSF grants DMS-??.}
\maketitle

\begin{abstract}
Previous work showed that all Bernoulli shifts over a free group are orbit-equivalent. This result is strengthened here by replacing Bernoulli shifts with the wider class of properly ergodic countable state Markov chains over a free group. A list of related open problems is provided.
\end{abstract}


\noindent
\tableofcontents

\section{Introduction}

Consider countable groups $\G,\La$, standard probability spaces $(X,\mu), (Y,\nu)$ and probability-measure-preserving (pmp) actions
$$\G \cc (X,\mu), \quad \La \cc (Y,\nu).$$
These actions are {\bf orbit-equivalent} (OE) if there exists a measure-space isomorphism $\Phi:(X,\mu) \to (Y,\nu)$ such that $\Phi( \G x) = \La x$ for a.e. $x$ (so $\Phi$ takes orbits to orbits). More generally, these actions are {\bf stably orbit-equivalent} (SOE) if there exist positive measure sets $X' \subset X$, $Y' \subset Y$ and a measurable isomorphism $\Phi:X' \to Y'$  such that $\Phi_*(\mu \resto X')$ is a scalar multiple of $\nu \resto Y'$ and $\Phi( \G x \cap X') = \La \Phi(x) \cap Y'$ for a.e. $x$. 

Dye proved that any two essentially free ergodic pmp actions of the  integers are OE \cite{MR0131516, MR0158048}. More generally, if $\G, \La$ are countably infinite amenable groups, then any two essentially free ergodic pmp actions of $\G, \La$ are OE by a theorem of Ornstein-Weiss \cite{OW80} (see also \cite{MR662736} for the non-singular case). On the other hand, when $\G$ is non-amenable, then Epstein showed that there exist uncountably many pairwise non-OE essentially free ergodic pmp actions of $\G$ \cite{epstein-oe, MR2529949}. This followed the work of many authors on various special cases (see \cite{MR2135736, MR2138136, ioana-oe2} for example).

Here we are motivated by the problem of classifying a special class of actions, called Bernoulli shifts, up to OE. Given a standard probability space $(K,\k)$, let $K^\G$ be the set of all functions $x:\G \to K$. We denote such a function by $x=(x_g)_{g\in \G} \in K^\G$. Then $\G$ acts on $K^\G$ by shifting $(g\cdot x)_f := x_{fg}$. This action preserves the product measure $\k^\G$. The system $\G \cc (K,\k)^\G$ is called the {\bf Bernoulli shift over $\G$  with base space $(K,\k)$}. Bernoulli shifts play a central role in the classification theory of measure-preserving actions \cite{bowen-survey}. 

It is a consequence of Popa's cocycle-super-rigidity Theorems \cite{popa-malleable, MR2231962} and Kida's OE-rigidity Theorems \cite{MR2680399, MR2369194} together with sofic entropy theory \cite{bowen-jams-2010} that there are many groups $\Ga$ with the property that if two Bernoulli shifts over $\Ga$ are OE then their base spaces have the same Shannon entropy. This is explained in more detail in \cite{bowen-jams-2010}. For such groups there is a continuum of pairwise non-OE Bernoulli shifts. 

On the other hand, free groups appear to be remarkably flexible. I showed in \cite{MR2763777} that all Bernoulli shifts over a non-abelian free group $\F$ are OE. Moreover, Bernoulli shifts over non-abelian free groups of different finite rank are stably-orbit-equivalent \cite{MR2763778} (see also \cite{MR3108102} for a nice exposition and further results). 
The main result of this paper is:

\begin{thm}\label{thm:main}
Let $\F$ be a non-abelian free group of finite rank. Then all properly ergodic countable-state Markov chains over $\F$ are OE. In particular, they are all OE to a Bernoulli shift over $\F$.
\end{thm}

\begin{remark}
An action is {\bf properly ergodic} if it is ergodic and there does not exist a co-null set on which the group acts transitively. Markov chains over free groups are carefully defined in \S \ref{sec:mc1}. 
\end{remark}

\begin{remark}
The space of measure-preserving actions of $\F$ on a standard probability space $(X,\mu)$ is denoted by $A(\F,X,\mu)$. It admits a natural Polish topology called the weak topology. Using \cite[Lemma 9.4]{bowen-entropy-2010a} it can be shown that the subset $\rm{MC} \subset A(\F,X,\mu)$ of actions that are measurably conjugate to a properly ergodic Markov chain is dense. More generally, in \cite{MR3420545} I showed that the OE-class of any essentially free action of $\F$ is weakly dense in the space of actions.
\end{remark}

\subsection{Questions and comments}

\begin{enumerate}
\item Is there a nice characterization of the measured equivalence relations that are SOE to a Bernoulli shift over $\F$? Such a characterization should help determine which of the following actions are SOE to a Bernoulli shift over $\F$: unimodular Galton-Watson trees \cite[Example 1.1]{aldous-lyons-unimodular}, Bernoulli shifts over surface groups or other treeable groups, Poisson point processes in the hyperbolic plane conditioned on the origin being contained in the point process, actions of the form $\G \cc \Aut(T_d)/\La$ where $\G,\La$ are lattices in $\Aut(T_d)$ and $\Aut(T_d)$ is the automorphism group of the $d>2$ regular tree, actions of $\F$ with completely positive Rokhlin entropy (for the definition of this see \cite{bowen-survey}), (orbit) factors of Bernoulli shifts, non-weakly-compact Gaussian actions of free groups, non-hyperfinite ergodic subequivalence relations of Bernoulli shifts over $\F$, inverse limits of Bernoulli shifts, Markovian planar stochastic hyperbolic infinite triangulations as in \cite{MR3520011}, the free spanning forest FSF of the Cayley graph of a surface group \cite{MR1825141}, and the cluster relation of Bernoulli percolation in the non-uniqueness phase of non-amenable Cayley graph \cite{MR1423907}.

\item  Weak compactness of actions was defined in \cite{MR2680430}. It is an SOE invariant and was shown in \cite{bowen-survey} to imply zero Rokhlin entropy. Therefore no weakly compact ergodic action of any group can be SOE to a Bernoulli shift over $\F$. For example, if $\F$ is embedded densely into a compact group $K$ then the translation action $\F \cc K$ is compact and therefore, weakly compact. 

\item Rigid actions were defined in \cite{popa-annals-2006} (see also \cite{ioana-sl2z} for an ergodic-theoretic formulation). From the proof of \cite[Proposition 3.3]{MR2568879}), it follows that no Bernoulli action admits a rigid orbit-factor. For example, the usual action of $\rm{SL}(2,\Z)$ on the 2-torus is rigid \cite{popa-annals-2006}. So the direct product of $\rm{SL}(2,\Z) \cc \TT^2$ with a Bernoulli action of $SL(2,\Z)$ cannot be OE to a Bernoulli shift (over any countable group). 

\item More generally, if $\F \cc (X,\mu)$ is OE to a Bernoulli shift over $\F$ then the group measure-space construction $L^\infty(X,\mu) \rtimes \F$ has Haagerup's property. This rules out the previous examples. It also rules out the following: suppose $\G$ is a group with property (T) and $\F \to \G$ is a surjection. Consider a generalized Bernoulli shift $\F \cc (K,\k)^\G$. The direct product of this action with a Bernoulli shift action of $\F$ is essentially free and ergodic but the group measure-space construction does not have the Haagerup property \cite[Theorem 0.2]{MR2833560}. So it cannot be OE to a Bernoulli shift over $\F$. (Thanks for Adrian Ioana for explaining the last three examples to me.)

\item There exist properly ergodic Markov chains whose Koopman representation is not contained in the countable sum of left-regular representations. There are also properly ergodic  Markov chains with zero Rokhlin entropy and other chains with negative $f$-invariant \cite{bowen-survey}. Since Bernoulli shifts do not have these properties, they cannot be OE invariants.

\item Bernoulli actions are {\bf solidly ergodic} in the sense that every subequivalence relation of the orbit relation decomposes into a hyperfinite piece and at most countably many strongly ergodic pieces \cite{MR2647134}. Because this property is an OE invariant, Theorem \ref{thm:main} implies properly ergodic Markov chains over $\F$ are solidly ergodic.

%\item It follows from Theorem \ref{thm:main} that properly ergodic Markov chains over $\F$ are neither weakly compact or rigid. Moreover, they must be strongly ergodic because strong ergodicity is an OE invariant and Bernoulli shifts are strongly ergodic \cite{MR2583950}. 

\item If $\F \cc (X,\mu)$ is solidly ergodic, essentially free, has no weakly compact orbit factors and its group measure space construction is Haagerup then is it OE to a Bernoulli shift?

\end{enumerate}

\subsection{Remarks on the proof}
There are two main parts: first we show that any properly ergodic Markov chain is OE to a Markov chain that is ``generator-ergodic'' in the sense that its symbolic restriction to any generator subgroup  is ergodic and essentially free (Proposition \ref{prop:reduction}). Second we show that every generator-ergodic Markov chain  is OE to a Bernoulli shift. 

The first step is by explicit construction, involving some ``edge-sliding'' arguments. In fact, the orbit-equivalences are continuous. The second step uses Dye's Theorem for actions of $\Z$ (as a black box) and so is considerably less constructive.

{\bf Acknowledgements}. Thanks to Brandon Seward, Robin Tucker-Drob and Peter Burton for discussing this problem with me. The picture greatly clarified from these discussions. Also thanks to Adrian Ioana for pointing me to \cite{MR2568879} and its implications.

\section{Preliminaries}

\subsection{Notation}
Let $S$ be a finite set and $\F=\langle S\rangle$ be the free group generated by $S$.  Let $A$ be a finite or countable set called the {\bf alphabet}. Then $A^\F$ is the set of all functions from $\F$ to $A$. We denote such a function by $x=(x_g)_{g\in \F}$. Let $\F$ act on $A^\F$ by
$$(g\cdot x)_f = x_{fg}.$$
This is called the {\bf shift action} and is denoted by $\F \cc A^\F$. Let $\Prob_\F(A^\F)$ denote the set of all shift-invariant Borel probability measures on $A^\F$. 

Similarly, define
$$T:A^\Z \to A^\Z, \quad (Tx)_n = x_{n+1}$$
and let $\Prob_\Z(A^\Z)$ denote the set of all $T$-invariant Borel probability measures on $A^\Z$.

If $(X,\mu)$ is a measure space, $Y$ is a Borel space and $\phi:X \to Y$ is measurable then the {\bf pushforward measure} $\phi_*\mu$ on $Y$ is defined by $\phi_*\mu(E) = \mu(\phi^{-1}(E))$ for measurable $E \subset Y$.

For $g \in \F$, let $\pi_g(x)=x_g$. Similarly, for $n\in \Z$ and $x\in A^\Z$, let $\pi_n(x)=x_n$. We will frequently abuse notation by writing $\pi$ for either $\pi_e$ or $\pi_0$ (depending on whether the argument is in $A^\F$ or $A^\Z$). 

\subsection{Markov Chains}\label{sec:mc1}

Here we define Markov chains over free groups. Let $|\cdot|$ denote the word length on $\F$. So for $g\in \F$, $|g|$ is the smallest integer $n$ such that $g$ is a product of $n$ elements in $S \cup S^{-1}$. 

\begin{defn}
For $s\in S \cup S^{-1}$, let
$$\past(s)=\{g\in \F:~ |gs^{-1}| = |g|-1 \}.$$ 
So $\past(s)$ consists of all reduced words that end in $s$. Note that $\F$ is the disjoint union of $\{e\}$ and $\past(s)$ for $s\in S \cup S^{-1}$.
\end{defn}

\begin{defn}\label{defn:past2}
Let $\mu \in \Prob_\F(A^\F)$ be a shift-invariant measure. For $s\in S$, $\mu$ is {\bf $s$-Markov} if the following is true. Let $\cF_s$ be the sigma-algebra of Borel subsets of $A^\F$ generated by the functions
$$x \mapsto x_g$$
for $g \in \past(s)$. Also let $\hat{\cF}_s$ be the sigma-algebra of Borel subsets of $A^\F$ generated by the functions
$$x \mapsto x_g$$
for $g \notin \past(s)$. Recall that $\pi:A^\F \to A$ is the time 0 map $\pi(x)=x_e$. Then $\mu$ is {\bf Markov} if $\cF_s$ is independent of $\hat{\cF}_s$ conditioned on $\pi$ with respect to $\mu$. Equivalently, if for every $E \subset \cF_s, \hE \subset \hat{\cF}_s$ and $a\in A$,
$$\mu(E \cap \hE| \pi = a) = \mu(E | \pi=a) \mu(\hE | \pi=a).$$
We say that $\mu$ is {\bf Markov with respect to $S$} if it is $s$-Markov for every $s\in S$. Usually we will simply say that $\mu$ is Markov if $S$ is understood. For example, Bernoulli shifts of the form $\F \cc (K,\k)^\F$ in which $K$ is a countable or finite set are Markov.
\end{defn}

\noindent {\bf Related literature}. Tree-indexed Markov chains were introduced in \cite{MR1258875, MR1254826}. The entropy theory of Markov chains over free groups is studied in \cite{bowen-entropy-2010a}.

%\begin{remark}
%There exists a measure $\mu \in \Prob_\F(A^\F)$ that is Markov with respect to some free generating set but not with respect to every free generating set. However, if $\mu$ is Markov with respect to $S$ and if $S$ is the disjoint union of sets $U$ and $V$ then $\mu$ is Markov with respect to $U^{-1} \cup V$. 
%\end{remark}

\section{General results regarding Markov chains}

This section establishes some general results on Markov chains. It also establishes the very useful Lemma \ref{lem:helper0} showing that if an action is Markov with respect to $|S|-1$ generators, and the restriction to the last generator is Markov, then the action itself is Markov. This will be used in both parts of the proof of the main theorem. 

\subsection{Cylinder sets}
To begin, we obtain a formula for Markov measures of cylinder sets.

\begin{defn}
The {\bf left-Cayley graph} of $\F$ has vertex set $\F$ and edge set $(g,sg)$ for $g\in \F, s\in S$. Because of the way we define the action $\F \cc A^\F$, the left-Cayley graph is more relevant to our concerns than the more usual right-Cayley graph. A subset $W \subset \F$ is {\bf left-connected} if its induced subgraph is connected, equivalently if for every $w_1,w_2\in W$ there exist $s_1,\ldots, s_n \in S \cup S^{-1}$ such that $w_2 = s_n\cdots s_1 w_1$ and $s_i\cdots s_1 w_1 \in W$ for all $1\le i \le n$. 
\end{defn}

\begin{defn}
Given $D \subset \F$ and $\phi:D \to A$, let 
$$\Cyl(\phi)=\{x\in A^\F:~x_g=\phi(g)~\forall g\in D\}$$
be the {\bf cylinder set} of $\phi$. A similar definition applies to $\Z$ in place of $\F$. 
\end{defn}

\begin{lem}\label{lem:basic-Markov}
Let $\mu \in \Prob_\F(A^\F)$ be a shift-invariant measure. Then $\mu$ is Markov if and only if for every left-connected finite set $D \subset \F$ such that $e\in D$ and every $\phi:D \to A$,
$$\mu(\Cyl(\phi)) = \mu( x_e=\phi(e)) \prod_{g\in D \setminus \{e\}} \mu\Big(x_g = \phi(g) | x_{\s(g)} = \phi\big(\s(g)\big)\Big)$$
where $\s(g) \in \F$ is the unique element satisfying $|\s(g)|=|g|-1$ and $g\s(g)^{-1}\in S \cup S^{-1}$. In other words, $\s(g)$ is on the unique path from $g$ to $e$ and $|\s(g)| = |g|-1$. 
\end{lem}

\begin{proof}
If $\mu$ satisfies the condition above then it is clearly Markov. So suppose that $\mu$ is Markov. We prove the formula above by induction on $|D|$. If $|D|=1$ then the statement is trivial. So suppose $|D|>1$. Then there exists $g\in D$ such that $D':=D \setminus \{g\}$ is left-connected. 

We will reduce to the special case in which $g\in S \cup S^{-1}$. To do this, define $h\in \F$ by: if $g=e$ then let $h \in (S\cup S^{-1}) \cap W$. Otherwise, set $h=\s(g)$. Because $\mu$ is shift-invariant,
$$\mu(h\Cyl(\phi)) = \mu(\Cyl(\phi)).$$
However,
$$h\Cyl(\phi) = \Cyl(\psi)$$
where $\psi:Dh^{-1} \to A$ is defined by $\psi(dh^{-1})=\phi(d)$. By choice of $h$, note that $gh^{-1} \in S\cup S^{-1}$ and therefore $\s(gh^{-1})=e$. After replacing $\phi$ with $\psi$ and $g$ with $gh^{-1}$, we see that it suffices to prove the claim when $g\in S \cup S^{-1}$ which we now assume.

Let $D' = D \setminus \{g\}$ and $\phi'$ be the restriction of $\phi$ to $D'$. By induction,
\begin{eqnarray*}
\mu(\Cyl(\phi))  &=& \mu(\Cyl(\phi')) \mu( x_g = \phi(g)| x \in \Cyl(\phi')) \\
&=& \mu( x_e=\phi(e)) \prod_{f\in D \setminus \{e,g\}} \mu\Big(x_f = \phi(f) | x_{\s(f)} = \phi\big(\s(f)\big)\Big) \mu\big( x_g = \phi(g)| x \in \Cyl(\phi')\big).
\end{eqnarray*}
Because $D' \setminus \{e\} \subset \F \setminus \past(g)$, the Markov property implies 
$$\mu\big( x_g = \phi(g)| x \in \Cyl(\phi')\big) = \mu\big(x_g=\phi(g)|x_e=\phi(e)\big) = \mu\Big(x_g=\phi(g)|x_{\s(g)}=\phi\big(\s(g)\big)\Big).$$
Combined with the previous equation, this completes the induction step.
\end{proof}

\subsection{Ergodicity and freeness}

In this subsection, we establish criteria for ergodicity and essential freeness of Markov chains.

\begin{defn}\label{defn:r}
Let $\mu \in \Prob_\F(A^\F)$ be a shift-invariant measure. For $s\in S$, define the {\bf symbolic restriction map}
$$R^s:A^\F \to A^\Z \quad R^s(x)_n = x_{s^n}$$
and $\mu_s \in \Prob_\Z(A^\Z)$ by $\mu_s = R^s_*\mu$. The measure $\mu_s$ is called the {\bf symbolic restriction of $\mu$ to the subgroup generated by $s$}. 
\end{defn}

\begin{defn}\label{defn:E}
Given $\mu \in \Prob_\F(A^\F)$ and $s\in S \cup S^{-1}$, let $E^\mu_s$ be the set of all $(a,b) \in A\times A$ such that
$$\mu_s\big(\{x\in A^\Z:~(x_0,x_1)=(a,b)\}\big)>0.$$
Let $E^\mu = \cup_{s\in S \cup S^{-1}} E^\mu_s$. Let $\cR^\mu_s, \cR^\mu \subset A \times A$ be the equivalence relation generated by $E^\mu_s, E^\mu$ respectively. A $\cR^\mu_s$-class $A' \subset A$ is called {\bf periodic} if every $a'\in A'$ has in-degree and out-degree 1 in the directed graph $(A,E^\mu_s)$. Otherwise $A'$ is called {\bf aperiodic}. 
\end{defn}

\begin{lem}\label{lem:reduction2}
Suppose $\mu \in \Prob_\F(A^\F)$ is Markov and $\pi_*\mu$ has full support on $A$. For any $s\in S$, $\mu_s$ is ergodic if and only if $\cR^\mu_s=A\times A$. Also $\mu_s$ is essentially free if and only if every $\cR^\mu_s$-class is aperiodic. Similarly, $\mu$ is ergodic if and only if $\cR^\mu=A\times A$. If $\mu$ is ergodic then $\mu$ is properly ergodic if and only if there exists some $s \in S$ such that some $\cR^\mu_s$-class is aperiodic.

% Then the number of ergodic components of $\mu_s$ $(\mu)$ is equal to the number of classes of $\cR^\mu_s$ ($\cR^\mu$). 
\end{lem}

\begin{proof}
The first statement is well-known. It follows, for example, from \cite[Section 6.1, page 338 in the 2nd edition]{durrett1996probability}. The second statement is a trivial exercise. The last two statements are similar.
\end{proof}

\subsection{A sufficient condition for a measure to be Markov}

\begin{lem}\label{lem:helper0}
For each $s\in S$, let $\nu_s \in \Prob_\Z(A^\Z)$ be a shift-invariant measure. Let $t \in S$. Suppose that $\nu_s$ is Markov for every $s\in S \setminus \{t\}$ and $\pi_*\nu_s = \pi_*\nu_r$ for every $r,s \in S$ where $\pi:A^\F \to A$ is the time 0 map $\pi(x)=x_e$.

Then there exists a unique shift-invariant measure $\rho \in \Prob_\F(A^\F)$ such that 
\begin{itemize}
\item $\rho_s = \nu_s$ for all $s \in S$,
\item $\rho$ is $s$-Markov for every $s \in S\setminus \{t\}$.
\end{itemize}
Moreover, if $\nu_t$ is also Markov then $\rho$ is Markov.
\end{lem}

\begin{proof}
Let $D \subset \F$ be finite, left-connected and satisfying $e\in D$. Let $\phi:D \to A$. For $g\in \F \setminus \{e\}$, define $\s(g) \in \F$ is as in Lemma \ref{lem:basic-Markov}. Set
$$\hD:=\{g\in D \setminus \{e\}:~ g\s(g)^{-1} \notin \{t,t^{-1}\}\}.$$
For $g \in D$, let 
$$D_g = \{ i \in \Z:~ t^i g \in D\}, \quad \phi_g: D_g \to A, ~\phi_g(i) = \phi(t^ig).$$
If $\rho$ exists then since it is $s$-Markov for every $s \in S \setminus \{t\}$, 
$$\rho(\Cyl(\phi)) = \nu_t\big(\Cyl(\phi_e)\big)\prod_{g \in \hD} \nu_{g\s(g)^{-1}}\Big(y_1 = \phi(g)|x_0 = \phi\big(\s(g)\big)\Big) \nu_t\big(\Cyl(\phi_g) | x_0 = \phi(g)\big).$$
This proves uniqueness. It also implies existence because we can define $\rho$ by the above equation since it satisfies the hypotheses of the Kolmogorov Extension Theorem. If $\nu_t$ is Markov, then Lemma \ref{lem:basic-Markov} implies $\rho$ is Markov.

\end{proof}

\section{General constructions of orbit-equivalences}

To prove the main theorem we will construct orbit-equivalences by first constructing alternative actions of $\F$ on $A^\F$ with the same orbits as the usual action. General facts regarding this construction are presented here.

\begin{lem}\label{lem:basic1}
Suppose $\tau:(S\cup S^{-1}) \times A^\F \to \F$ is a function satisfying
$$\tau(s^{-1}, \tau(s,x) \cdot x)  = \tau(s,x)^{-1} \quad \forall s\in S\cup S^{-1}, x\in A^\F.$$
There there exists an action $\ast:\F \times A^\F \to A^\F$ of $\F$ satisfying
$$s\ast x = \tau(s)\cdot x \quad \forall s\in S\cup S^{-1}, x\in A^\F$$
and a function $\omega: \F \times A^\F \to \F$ extending $\tau$ and satisfying the cocycle equation
$$\omega(gh,x) = \omega(g,h\ast x) \omega(h,x).$$
Moreover, $g \ast x = \omega(g,x)\cdot x$. Also, if $\Omega:A^\F \to A^\F$ is defined by
$$\Omega(x)_h = x_{\omega(h,x)}$$
then $\Omega$ is $(\ast,\cdot)$-equivariant in the sense that
$$g\cdot (\Omega x) = \Omega(g\ast x).$$
\end{lem}

\begin{proof}
The existence of $\ast$ is immediate since $\F$ is freely generated as a semi-group by $S \cup S^{-1}$ and 
$$s^{-1}\ast (s \ast x) = \tau(s^{-1}, \tau(s,x) \cdot x)  \cdot (\tau(s,x)\cdot x)  = x.$$
Similarly the existence of $\omega$ is immediate and the equation $g \ast x = \omega(g,x)\cdot x$ follows from the cocycle equation by inducting on $|g|$. 

To see that $\Omega$ is $(\ast,\cdot)$-equivariant, let $g,h\in \F$. Then
\begin{eqnarray*}
\Omega (g\ast x)_h &=&   (g\ast x)_{\omega(h, g\ast x)} = (\omega(g,x) \cdot x)_{\omega(h, g\ast x)} = \\
&=& x_{\omega(h, g\ast x) \omega(g,x)} =  x_{\omega(hg,x)} = (\Omega x)_{hg} = (g\cdot \Omega x)_h.
 \end{eqnarray*}
\end{proof}

\begin{lem}\label{lem:basic2}
Let $\ast, \omega,\Omega$ be as in Lemma \ref{lem:basic1}. Also let $s\in S$. If for every $x\in A^\F$ 
$$g \in \past(s) \Leftrightarrow \omega(g,x) \in \past(s)$$
and $\mu \in \Prob_\F(A^\F)$ is $s$-Markov then $\Omega_*\mu$ is $s$-Markov.
\end{lem}

\begin{proof}

Define  $\cF_s,\hcF_s$ as in Definition \ref{defn:past2}.

\noindent {\bf Claim 1}. $\Omega^{-1}(\cF_s) \subset \cF_s$.

\begin{proof}
For $g\in \F$, let $\pi_g:A^\F \to A$ be the coordinate function $\pi_g(x)=x_g$. Because $\cF_s$ is generated by sets of the form $\pi_g^{-1}(B)$ for $B\subset A$ and $g\in \past(s)$, $\Omega^{-1}(\cF_s)$ is generated by sets of the form $\Omega^{-1}\pi_g^{-1}(B) = (\pi_g \circ \Omega)^{-1}(B)$. So it suffices to show that if $g\in \past(s)$ then $\pi_g \circ \Omega$ is $\cF_s$-measurable. Since $\pi_g \Omega x = (\Omega x)_g = x_{\omega(g,x)}$ and $\omega(g,x) \in \past(s)$, it follows that $\pi_g\Omega$ is $\cF_s$-measurable.
\end{proof}

A similar argument shows that $\Omega^{-1}(\hcF_s) \subset \hcF_s$. To prove $\Omega_*\mu$ is $s$-Markov, let  $E_1 \in \cF_s$ and $E_2 \in \hcF_s$. Claim 1 implies $\Omega^{-1}(E_1) \in \cF_s$ and $\Omega^{-1}(E_2) \in \hcF_s$. Since $\mu$ is $s$-Markov and $\pi_e \Omega = \pi_e$, for any $a\in A$,
\begin{eqnarray*}
\Omega_*\mu(E_1 \cap E_2 | \pi_e = a) &=& \mu(\Omega^{-1}(E_1) \cap \Omega^{-1}(E_2) | \pi_e = a) \\
&=& \mu(\Omega^{-1}(E_1)  | \pi_e = a) \mu( \Omega^{-1}(E_2) | \pi_e = a) \\
&=& \Omega_*\mu(E_1|\pi_e = a) \Omega_*\mu(E_2|\pi_e = a).
\end{eqnarray*}
Since $E_1,E_2,a$ are arbitrary this shows $\Omega_*\mu$ is $s$-Markov.
\end{proof}

\begin{lem}\label{lem:basic3}
Let $\ast, \omega,\Omega$ be as in Lemma \ref{lem:basic1}. Suppose there is another function $\hat{\tau}:(S\cup S^{-1}) \times A^\F \to \F$ satisfying
$$\hat{\tau}(s^{-1}, \hat{\tau}(s,x) \cdot x)  = \hat{\tau}(s,x)^{-1}$$
for every $s\in S\cup S^{-1}$. Let $\star: \F \times A^\F \to \F$, $\homega:\F \times A^\F \to \F, \hat{\Omega}:A^\F \to A^\F$ denote the associated action, cocycle and map as in Lemma \ref{lem:basic1}. Suppose as well that $\omega(\homega(s,\Omega x),x)=s$ for every $s\in S\cup S^{-1}$. Then 
\begin{eqnarray}\label{inverse00}
\omega(\homega(g,\Omega x),x)=g ~\forall g\in \F
\end{eqnarray}
and $\hat{\Omega}\Omega x = x$ for all $x$.
\end{lem}

\begin{proof}
To prove the first claim, it suffices to prove: if $g_1,g_2 \in\F$ satisfy (\ref{inverse00}) for all $x\in A^\F$ then the product $g_1g_2$ also satisfies  (\ref{inverse00}). This follows from the cocycle equations 
\begin{eqnarray*}
\omega(\homega(g_1g_2,\Omega x), x) &=& \omega(\homega(g_1,g_2 \star \Omega x) \homega(g_2,\Omega x), x) \\
&=& \omega(\homega(g_1,g_2 \star \Omega x), \homega(g_2,\Omega x)\ast x)  \omega(\homega(g_2,\Omega x), x)
\end{eqnarray*}
Since $\omega(\homega(g_2,\Omega x), x)=g_2$ it suffices to show  $\omega(\homega(g_1,g_2 \star \Omega x), \homega(g_2,\Omega x)\ast x)=g_1$. This will follow from the assumption that $g_1$ satisfies (\ref{inverse00}) once we show that 
$$g_2 \star \Omega x = \Omega(\homega(g_2,\Omega x)\ast x).$$ 
This follows from Lemma \ref{lem:basic1} and since $\Omega$ is $(\ast,\cdot)$-equivariant:
$$g_2 \star \Omega x = \homega(g_2, \Omega x) \cdot (\Omega x) = \Omega( \homega(g_2,\Omega x) \ast x).$$
The proves the first claim. To prove the last, let $x\in A^\F$ and $h\in \F$. Then
\begin{eqnarray*}
(\hat{\Omega} \Omega x)_h &=& (\Omega x)_{\homega(h,\Omega x)} = x_{\omega(\homega(h,\Omega x),x)} = x_h.
\end{eqnarray*}
\end{proof}

\section{From properly ergodic to generator-ergodic}

The main result of this section is:
\begin{prop}\label{prop:reduction}
Let $\mu \in \Prob_\F(A^\F)$ be properly ergodic and Markov. Then there exists a countable set $B$ and a shift-invariant measure $\rho \in \Prob_\F(B^\F)$ such that $\F \cc (A^\F,\mu)$ is OE to $\F \cc (B^\F,\rho)$, $\rho$ is Markov and for every $s\in S$, $\rho_s$ is essentially free and ergodic.
\end{prop}

To prove this result, we will construct a very specific kind of orbit-equivalence which we then apply multiple times with slightly varying hypotheses. The orbit-equivalence we build depends on a choice of a subset $\cE \subset E^\mu_u$ (where $u\in S$ and $E^\mu_u$ is as in Definition \ref{defn:E}) satisfying some technical conditions described next.

\begin{defn}
Let $\mu \in \Prob_\F(A^\F)$ and $s\in S$. A subset $\cE\subset E^\mu_s$ is said to be {\bf $\mu_s$-special} if 
\begin{enumerate}
\item for every $a\in A$ there does not exist $b,c \in A$ such that both $(a,b)\in \cE$ and $(c,a)\in \cE$, and
\item if $(a,b)\in \cE$ then the $\cR^\mu_s$ classes of $a$ and $b$ are aperiodic.
\end{enumerate}
%For example, if $\cE=\{(a,b)\}$ with $a\ne b, (a,b) \in E^\mu_s$ then $\cE$ is $\mu_s$-special.

%In other words, the directed graph $(A,\cE)$ has the property that for every vertex $a\in A$, either the in-degree or the out-degree of $a$ in $(A,\cE)$ is zero.
\end{defn}

The next result is the key lemma towards proving Proposition \ref{prop:reduction}. We will apply it multiple times to obtain Proposition \ref{prop:reduction}. The reader who is only interested in the special case in which the alphabet $A$ is finite can assume that $\cE=\{(a,b)\}$ is a singleton. 

\begin{lem}\label{lem:key0}
Let $\mu \in \Prob_\F(A^\F)$ be Markov, properly ergodic such that $\pi_*\mu \in \Prob(A)$ is fully supported. Let $u,t \in S$ be distinct and let $\cE$ be $\mu_u$-special.  Then there exists a Markov measure $\rho \in \Prob_\F(A^\F)$ such that 
\begin{itemize}
\item $\F \cc (A^\F,\mu)$ is OE to $\F \cc (A^\F,\rho)$, 
\item $\mu_s=\rho_s$ for all $s \in S\setminus \{t\}$,
\item $\cR^\rho_t \supset \cE \cup \cR^\mu_t$,
\item for every $(a,b)\in \cE$, the $\cR^\rho_t$-classes of $a$ and $b$ are aperiodic.
\end{itemize}
\end{lem}

\begin{proof}[Proof of Lemma \ref{lem:key0}]
Let $(a,b) \in \cE$. Because the $\cR^\mu_u$-class of $b$ is aperiodic, there exist a smallest number $n=n(b)>0$, elements $b_0,\ldots, b_n \in A$ and $\eta(b) \in A$ such that
\begin{itemize}
\item $b=b_0$,
\item $(b_i,b_{i+1}) \in E^\mu_u$ for all $0\le i < n$
\item $b_n \ne \eta(b)$ and $(b_{n-1},\eta(b)) \in E^\mu_u$. 
\end{itemize}
Choose a function $\eta$ satisfing the above. Let
$$\cF=\big\{ (a,b,\eta(b)):~ (a,b) \in \cE\big\}.$$
%For $x\in A^\F$, set $\eta(x)=\eta(x_e)$ whenever this is well-defined. 
Let
$$F(x) := (x_{u^{-1}}, x_e, x_{u^{n(x_e)}})$$
whenever $(x_{u^{-1}}, x_e) \in \cE$.

Define $\tau:(S\cup S^{-1}) \times A^\F \to \F$ by
$$\tau(s,x) = s  \quad \forall s\in S\cup S^{-1}\setminus \{t,t^{-1}\},$$
\begin{displaymath}
\tau(t,x) = \left \{ \begin{array}{cc}
ut& \textrm{ if } F(ut\cdot x) \in \cF  \\  
u^{-1}t & \textrm{ if } F(t\cdot x) \in \cF  \\  
t & \textrm{ otherwise } \end{array}\right.\end{displaymath}

\begin{displaymath}
\tau(t^{-1},x) = \left \{ \begin{array}{cc}
(ut)^{-1} & \textrm{ if } F(x) \in \cF \\  
(u^{-1}t)^{-1} & \textrm{ if } F(u\cdot x) \in \cF \\  
t^{-1} & \textrm{ otherwise } \end{array}\right.\end{displaymath}
For example, $F(ut\cdot x)\in \cF$ means that $F(ut \cdot x)$ is well-defined and $F(ut\cdot x)\in \cF$.

Because $\cE$ is $\mu_u$-special, $\tau$ is well-defined (for example, it cannot be that $F(ut \cdot x)\in \cF$ and $F(t\cdot x) \in \cF$). Also $\tau$ satisfies the hypotheses of Lemma \ref{lem:basic1}. Let $\ast,\omega,\Omega$ be as in Lemma \ref{lem:basic1}.

We will show that $\F \cc^\ast (A^\F,\mu)$ has the same orbits as $\F \cc (A^\F,\mu)$ (modulo $\mu$ null sets) and $\F \cc^\ast (A^\F,\mu)$ is measurably-conjugate to an action of the form  $\F \cc (A^\F,\rho)$ where $\rho$ satisfies the conclusion. 

\noindent {\bf Claim 1}. For $\mu$-a.e. $x$, $\F \cdot x = \F \ast x$.

\begin{proof}
It is immediate that $\F \cdot x \supset \F \ast x$ for a.e. $x$. To show the opposite inclusion, it suffices to show that for a.e. $x \in A^\F$ and every $g \in \F$ there exists $h \in \F$ such that $\omega(h,x)=g$. Because of the cocycle equation, it suffices to prove this for $g\in S \cup S^{-1}$. The special case of $g \in S\cup S^{-1} \setminus \{t,t^{-1}\}$ is clear since in that case $\omega(g,x)=g$. 

We claim:
\begin{eqnarray*}
F(ut\cdot x) \in \cF &\Rightarrow & \omega(u^{-1}t,x)=  t  \label{thing1} \\
F(t\cdot x) \in \cF &\Rightarrow & \omega(ut,x)=  t  \label{thing2} \\
F(ut\cdot x) \notin \cF \wedge F(t\cdot x) \notin \cF &\Rightarrow & \omega(t,x)=  t. \label{thing3}
\end{eqnarray*}
To see the first equation, assume $F(ut\cdot x) \in \cF $. By the cocycle equation
$$ \omega(ut,x) = \omega(u^{-1},t\ast x) \omega(t,x) = u^{-1}(ut)=t.$$
The other cases are similar. This shows for a.e. $x\in A^\F$, there exists $h$ such that $\omega(h,x)=t$. The statement with $t^{-1}$ in place of $t$ is similar.
\end{proof}

\noindent {\bf Claim 2}. For any $x\in A^\F$ and $g\in \F$,
$$\omega(\omega(g,\Omega x), x) = g$$
and $\Omega (\Omega x) =x$.

\begin{proof}
The first claim is immediate for $g \in S\cup S^{-1} \setminus \{t,t^{-1}\}$. To handle the case $g=t$, suppose that $F(ut\cdot x) \in \cF $. Then $\omega(t,x)=ut$ and for any $m\in \Z$,
$$\omega(u^mt,x) = \omega(u^m, t\ast x) \omega(t,x) = u^m(ut)=u^{m+1}t.$$ 
So
$$(t \cdot \Omega x)_{u^m}= (\Omega x)_{u^mt} = x_{\omega(u^mt,x)} = x_{u^{m+1}t}.$$
Let $n = n(x_{ut})=n((t\cdot \Omega x)_e)$. Then
$$F(t\cdot \Omega x) = ( (t\cdot \Omega x)_{u^{-1}}, (t\cdot \Omega x)_e, (t\cdot \Omega x)_{u^n}) = (x_t,x_{ut}, x_{u^{n+1}t}) = F(ut\cdot x)\in \cF.$$
So $\omega(t,\Omega x)=u^{-1}t$ and
$$\omega(\omega(t,\Omega x), x)  = \omega(u^{-1}t,x) = t.$$
The other cases are similar. Claim 2 now follows from Lemma \ref{lem:basic3}. 
 
 \end{proof}

It now suffices to prove that if $\rho:=\Omega_*\mu$ then $\rho$ satisfies the conclusions of this lemma. Claims 1 and 2 show that $\F \cc (A^\F,\mu)$ is OE to $\F \cc (A^\F,\rho)$. The next three claims show that $\rho$ is Markov.

\noindent {\bf Claim 3}. If $s\in S\cup S^{-1} \setminus \{t^{-1},u, u^{-1}\}$ and $g \in \past(s)$ then $\omega(g,x) \in \past(s)$ for all $x\in A^\F$.

\begin{proof}
We prove this by induction on $|g|$. If $|g|=1$ then $g=s$ and either $s \ne t$ in which case $\omega(g,x)=g \in \past(s)$, or $s=t$ and $\omega(g,x) \in \{t,ut,u^{-1}t\} \subset \past(s)$. 

So assume $|g|>1$. Then we can write $g=hk$ for some $h \in S\cup S^{-1}, k \in \past(s)$ such that $|k| < |g|$. By induction we can assume $\omega(k,x) \in \past(s)$. Then
$$\omega(hk,x) = \omega(h, k\ast x)\omega(k,x).$$
To obtain a contradiction, suppose that $\omega(hk,x) \notin \past(s)$. 

Since $h \in S\cup S^{-1}$, $|\omega(h,k\ast x)| \in \{1,2\}$. If $|\omega(h,k\ast x)|=1$ then since $f \past(s) \subset \past(s)$ for all $f \in S \cup S^{-1} \setminus \{s^{-1}\}$ it must be that $\omega(h,k\ast x) = s^{-1}$. Since $s^{-1} \past(s) = \{e\} \cup \past(s)$, it must be that $\omega(hk,x)=e$ (since $\omega(\cdot, x)$ is injective) which implies $hk=g=e$, contradicting that $g\in \past(s)$. 

So suppose $|\omega(h,k\ast x)|=2$. Since $|h|=1$ this implies  $h \in \{t,t^{-1}\}$ and $\omega(h,k\ast x) \in \{ ut,u^{-1}t, (ut)^{-1}, (u^{-1}t)^{-1}\}$. If $f \in \F$ is any element with $|f|=2$ then
$$f \past(s) \subset \past(s) \cup \{e\} \cup S \cup S^{-1}.$$
If $\omega(h,k\ast x)\omega(k,x)=e$ then $hk=g=e$, contradicting that $g\in \past(s)$. So we may assume that 
\begin{eqnarray}\label{obvious}
\omega(h,k\ast x)\omega(k,x)\in S \cup S^{-1} \setminus \past(s).
\end{eqnarray}
Since $\omega(h,k\ast x) \in \{ ut,u^{-1}t, (ut)^{-1}, (u^{-1}t)^{-1}\}$ and $\omega(k,x) \in \past(s)$, this implies that $s\in \{u,u^{-1},t,t^{-1}\}$. By assumption $s\notin \{t^{-1},u, u^{-1}\}$. So $s=t$ and $\omega(k,x) \in \{t,ut,u^{-1}t\}$. Thus
$$\omega(g,x)=\omega(h,k\ast x)\omega(k,x) \in \{ ut,u^{-1}t, (ut)^{-1}, (u^{-1}t)^{-1}\} \{t,ut,u^{-1}t\}  \subset \past(t) \cup \{e\}.$$
Since we are assuming $\omega(g,x) \notin \past(t)$, this implies $\omega(g,x)=e$. Since $\omega(\cdot, x)$ is injective, this implies $g \in \{u,u^{-1}\}$, contradicting that $g\in \past(t)$. 

\end{proof}

\noindent {\bf Claim 4}. If $s\in S\cup S^{-1} \setminus \{t^{-1},u, u^{-1}\}$ and $g \notin \past(s)$ then $\omega(g,x) \notin \past(s)$ for all $x\in A^\F$.

\begin{proof}
To obtain a contradiction, suppose there exists $g\notin \past(s)$ and $y \in A^\F$ such that $\omega(g,y) \in \past(s)$. Let $\Omega y = x$. By Claim 2, $\Omega x = y$ and
$$ \omega(\omega(g, y), x) = \omega(\omega(g, \Omega x), x)  = g.$$
Since $\omega(g,y) \in \past(s)$, Claim 3 implies $\omega(\omega(g, y), x) = g  \in \past(s)$. This contradiction proves Claim 4.
\end{proof}

\noindent {\bf Claim 5}. 
$\rho$ is Markov.

\begin{proof}
It follows from Claims 3 and 4 that if $s \in S \cup S^{-1} \setminus \{t^{-1},u, u^{-1}\}$  and $x\in A^\F$ then
$$g \in \past(s) \Leftrightarrow \omega(g,x) \in \past(s).$$
By Lemma \ref{lem:basic2}, $\rho$ is $s$-Markov for every $s \in S\setminus \{u\}$. Since $\rho_u=\mu_u$ is also Markov, Lemma \ref{lem:helper0} implies $\rho$ is Markov.
\end{proof}

Let $x\in A^\F$, $s\in S \setminus \{t\}$ and $n\in \Z$. Then
$$(\Omega x)_{s^n} = x_{\omega(s^n,x)} =x_{s^n}.$$ 
Thus $R^s\Omega x = R^sx$ which implies $\rho_s = \mu_s$ for all $s\in S \setminus \{t\}$.

\noindent {\bf Claim 6}. 
$\cR^\mu_t \subset \cR^\rho_t$.

\begin{proof}
Let $(\a,\b) \in E^\mu_t$. We will show that $(\a,\b) \in E^\rho_t$. Indeed,
$$\mu\big(\{ x\in A^\F:~x_e = \a, x_t =\b, F(ut\cdot x) \notin \cF \textrm{ and }  F(t\cdot x) \notin \cF\}\big) >0.$$
This is because the event $ut\cdot x \in \cF$ depends only on $x_e, x_t$ and $x_{u^nt}$ where $n=n(x_{ut})$. Since $\mu$ is Markov, the event that $ut\cdot x \in \cF$ given that $x_t = \b$ does not depend on $x_e$. A similar statement hold for the event $F(t\cdot x) \in \cF$. Moreover, because $\cE$ is $\mu_u$-special, depending only on $\b$, one of the events $F(ut\cdot x) \notin \cF$, $F(t\cdot x) \notin \cF$ must occur. 

Now suppose that $x\in A^\F$ satisfies $x_e = \a, x_t =\b, F(ut\cdot x) \notin \cF, \textrm{ and }  F(t\cdot x) \notin \cF$. Then $\omega(t,x)=t$,  $(\Omega x)_e=\a$ and 
$$(\Omega x)_t = x_{\omega(t,x)} = x_t =\b.$$
Thus shows 
$$\Omega^{-1}\big(\{y \in A^\F:~ y_e=\a, y_t= \b\}\big) \supset \{ x\in A^\F:~x_e = \a, x_t =\b, F(ut\cdot x) \notin \cF, \textrm{ and }  F(t\cdot x) \notin \cF\}.$$
Therefore, 
$$\Omega_*\mu\big(\{y\in A^\F:~ y_e=\a, y_t=\b\}\big)>0.$$
Since $\rho=\Omega_*\mu$, this implies $(\a,\b) \in E^\rho_t$. Since $(\a,\b)$ is arbitrary, $E^\mu_t \subset E^\rho_t$. Since $\cR^\mu_t$ is generated by $E^\mu_t$, it follows that $\cR^\mu_t \subset \cR^\rho_t$.
\end{proof}

\noindent {\bf Claim 7}. Let $(a,b) \in \cE$ and let $\a\in A$ be such that $(\a,a) \in E^\mu_t$. Then $(\a,b) \in E^\rho_t$. 

\begin{proof}
 Since $(a,b) \in \cE \subset E^\mu_u$, the Markov property (via Lemma \ref{lem:basic-Markov}) implies
$$\mu(\{x\in A^\F:~x_e = \a, x_t = a, x_{ut}= b\})>0.$$
In fact,
$$\mu(\{x\in A^\F:~x_e = \a, x_t = a, x_{ut}= b, F(ut \cdot x) \in \cF\})>0.$$
This is because the event $F(ut \cdot x) \in \cF$ given $x_t=a, x_{ut}=b$ depends only on $x_{u^{n+1}t}$ where $n=n(b)$.

If $x\in A^\F$ is such that $x_e = \a, x_t = a, x_{ut}= b, F(ut \cdot x) \in \cF$ then $\omega(t,x)=ut$. So
$$(\Omega x)_e = x_e = \a, \quad (\Omega x)_t  = x_{\omega(t,x)} = x_{ut} = b.$$
So
$$\Omega^{-1}(\{y \in A^\F:~ y_e=\a, y_t= b\}) \supset \{x \in A^\F:~ x_e = \a, x_t = a, x_{ut}= b, F(ut \cdot x) \in \cF\}.$$
Therefore
$$\Omega_*\mu(\{y \in A^\F:~ y_e=\a, y_t= b\})>0.$$
This shows $(\a,b) \in E^\rho_t$.  %Since $(\a,a) \in \cR^\rho_t$ by Claim 6 and $\cR^\rho_t$ is an equivalence relation, it follows that $(a,b) \in \cR^\rho_t$. 
\end{proof}

It follows from Claim 6 and 7 that $(a,b) \in \cR^\rho_t$ for every $(a,b) \in \cE$. Therefore $\cE \subset \cR^\rho_t$. By Claim 6, $\cE \cup \cR^\mu_t \subset \cR^\rho_t$.

Let $a,b,\a$ be as in Claim 7 and observe that $(\a,a), (\a,b) \in E^\rho_t$ by Claims 6 and 7. Since $a\ne b$ (because $\cE$ is special) the out-degree of $\a$ in the directed graph $(A, E^\rho_t)$ is at least 2. So the $\cR^\rho_t$-class of $\a$ is aperiodic. Since $a$ and $b$ are $\cR^\rho_t$-equivalent to $\a$,  the $\cR^\rho_t$-classes of $a$ and $b$ are aperiodic. Since $(a,b) \in \cE$ is arbitrary, this finishes the lemma.

\end{proof}

\begin{proof}[Proof of Proposition  \ref{prop:reduction}]
Without loss of generality we may assume $\pi_*\mu$ is a fully supported measure on $A$. Because $\mu$ is properly ergodic, Lemma \ref{lem:reduction2} implies there exist $a\in A$ and $u \in S$ such that the $\cR^\mu_u$ class of $a$ is aperiodic. Let $[a]^\mu_u$ denote the $\cR^\mu_u$-class of $a$. Let $\cT_u \subset E^\mu_u$ be a spanning tree of the induced subgraph of $[a]^\mu_u$ in $(A,E^\mu_u)$. Because trees are bi-partitite, there exists a partition $A_0 \sqcup A_1$ of $[a]^\mu_u$ such that 
$$\cT_u \subset (A_0 \times A_1) \cup (A_1 \times A_0).$$
Let 
$$\cE_1 = \cT_u \cap (A_0 \times A_1), \quad \cE_2 = \cT_u \cap (A_1 \times A_0).$$
Then each $\cE_i$ is $\mu_u$-special. After applying Lemma \ref{lem:key0} successively using $\cE_1,\cE_2$ and letting $t$ vary over $S \setminus \{u\}$, we obtain the existence of a Markov measure $\rho \in \Prob_\F(A^\F)$ such that 
\begin{itemize}
\item $\F \cc (A^\F,\mu)$ is OE to $\F \cc (A^\F,\rho)$, 
\item for every $s\in S$, $\cR^\rho_s \supset \cR^\mu_s \cup \cT_u$,
\item for every $s\in S$, the $\cR^\rho_s$-class of $a$ is aperiodic.
\end{itemize}
Thus after replacing $\mu$ with $\rho$ if necessary, we may assume that the $\cR^\mu_s$-class of $a$ is aperiodic for every $s$. 

We can now apply the same argument as above for any $s\in S$ in place of $u$. Thus we obtain the existence of a Markov measure $\rho \in \Prob_\F(A^\F)$ such that 
\begin{enumerate}
\item $\F \cc (A^\F,\mu)$ is OE to $\F \cc (A^\F,\rho)$, 
\item for every $s,u\in S$, $\cR^\rho_s \supset \cR^\mu_s \cup \cT_u$,
\item for every $s\in S$, the $\cR^\rho_s$-class of $a$ is aperiodic.
\end{enumerate}
It follows from item (2) that $\cR^\rho_s \ni (a,b)$ for every $b$ such that there exists some $u\in S$ with $(a,b) \in \cR^\mu_u$. This is because $\cT_u$ generates the $\cR^\mu_u$-class of $a$. However, since $\mu$ is properly ergodic and $\pi_*\mu$ is fully supported, this implies $\cR^\rho_s \ni (a,b)$ for every $b \in A$. So $\cR^\rho_s = A \times A$. Thus $\rho_s$ is ergodic and by (3) essentially free. 
\end{proof}

\section{Proof of the main theorem}

The main theorem is obtained by applying a specific kind of orbit-equivalence to a given Markov system multiple times. This kind of orbit-equivalence does not change the 1-dimensional marginal $\pi_*\mu$ and preserves the Markov property. At the same time, it replaces one of the symbolic restrictions $\mu_t$ with a Bernoulli measure. 

To build these orbit equivalence, we will first need some well-known facts about full groups of measured equivalence relations (Definition \ref{defn:full}, Lemma \ref{lem:kakutani}). We then apply these facts to obtain a slightly enhanced version of Dye's Theorem (Lemma \ref{lem:key}). Then Lemma \ref{lem:main} establishes the specific kind of orbit-equivalence we need to prove the main theorem.

\begin{defn}\label{defn:full}
Recall that $T:A^\Z \to A^\Z$ is defined by $(Tx)_n = x_{n+1}$ and $\Prob_\Z(A^\Z)$ is the space of $T$-invariant Borel probability measures on $A^\Z$. For $\mu \in \Prob_\Z(A^\Z)$ let $[T,\mu]$ denote the {\bf full group} of the orbit-equivalence relation of $T$ modulo $\mu$. To be precise, $[T,\mu]$ consists of all measurable automorphisms $S: X \to X$ (where $X \subset A^\Z$ is $\mu$-conull and $T$-invariant) such that for every $x\in X$ there exists $n\in \Z$ with $Sx=T^nx$. Two such automorphisms are identified if they agree on a $\mu$-conull set.
\end{defn}

\begin{lem}\label{lem:kakutani}
Let $\mu \in \Prob_\Z(A^\Z)$ be ergodic and essentially free. Let $B$ be a finite or countable set and let $\phi:A^\Z \to B, \psi:A^\Z \to B$ be measurable maps with the same pushforward measures (so $\phi_*\mu=\psi_*\mu$). Then there exists $S \in [T,\mu]$ such that $\psi = \phi \circ S$.
\end{lem}

\begin{proof}
This result is well-known but I did not find a suitable reference (it partially generalizes a lemma in \cite{MR0369658}). 

Let $\{n_i\}_{i=1}^\infty = \Z$ be an enumeration of the integers. Let
$$X_1 = \big\{x\in A^\F:~ \psi(x) = \phi( T^{n_1}x) \big\}.$$
Define $S_1:X_1 \to A^\F$ by $S_1(x)=T^{n_1}x$. If $X_k$ and $S_k$ have been defined, let $X_{k+1}$ be the set of all $x\in A^\F \setminus \bigcup_{i=1}^k X_i$ such that 
$$\psi(x)=\phi(T^{n_{k+1}}x) \textrm{ and } T^{n_{k+1}}x \notin \bigcup_{i=1}^k S_i(X_i).$$
Define $X=\cup_k X_k$ and $S:X \to A^\F$ by $Sx = S_kx$ for $x\in X_k$. Because the $X_k$'s are pairwise disjoint, $S$ is well-defined. 

By ergodicity, for each $b\in B$, $\phi^{-1}(b) \subset \cup_{n\in \Z} T^n \psi^{-1}(b)$ modulo $\mu$-null sets. Therefore, $X$ is $\mu$-conull. By design, $\psi = \phi \circ S$. Moreover, since the $S_k(X_k)$'s  are pairwise disjoint and each $S_k$ is injective, $S$ is invertible. This shows $S$ is in the full group $[T,\mu]$.

\end{proof}

\begin{lem}\label{lem:key}
Let $\mu, \nu \in \Prob_\Z(A^\Z)$ be shift-invariant, ergodic, essentially free measures. Let $\pi:A^\Z \to A$ be the time 0 map. Suppose $\pi_*\mu=\pi_*\nu$ (in other words, for every $a\in A$, 
$$\mu\big(\{x\in A^\Z:~x_e=a\}\big) = \nu\big(\{x\in A^\Z:~x_e=a\}\big).$$
Then there exists an orbit equivalence $\Psi:A^\Z \to A^\Z$
from the shift action $\Z \cc^T (A^\Z,\mu)$ to $\Z \cc^T (A^\Z,\nu)$ such that $\pi = \pi \Psi.$
\end{lem}

\begin{proof}
By Dye's Theorem, there exists an orbit-equivalence $\Psi': A^\Z \to A^\Z$ from the shift action $\Z \cc^T (A^\Z,\mu)$ to $\Z \cc^T (A^\Z,\nu)$. By Lemma \ref{lem:kakutani} there exists $S \in [T,\mu]$ such that $\pi = \pi \Psi' S.$

Let $\Psi = \Psi' S$. Then $\Psi$ is an orbit-equivalence from $\Z \cc^T (A^\Z,\mu)$ to $\Z \cc^T (A^\Z,\nu)$ since pre-composing with an element of the full group does not change orbits. Also $\pi = \pi \Psi$.
\end{proof}

\begin{lem}\label{lem:main}
Let $\mu \in \Prob_\F(A^\F)$ be Markov. Let $t \in S$ and suppose that the symbolic restriction $\mu_{t}\in \Prob_\Z(A^\Z)$ is such that $\Z \cc^T (A^\Z,\mu_{t})$ is essentially free and ergodic. Also let $\nu \in \Prob_\Z(A^\Z)$ be an ergodic, essentially free, shift invariant measure. By abuse of notation, let $\pi:A^\F \to A$ denote the map $\pi(x)=x_e$ and let $\pi:A^\Z \to A$ denote the map $\pi(x)=x_0$. Suppose that $\pi_*\mu = \pi_*\nu$. 

Then the action $\F \cc (A^\F,\mu)$ is OE to $\F \cc (A^\F,\rho)$ where $\rho$ is a shift-invariant measure uniquely determined by the following. 
\begin{itemize}
\item $\rho_t = \nu$,
\item $\rho_s = \mu_s$ for all $s \in S \setminus \{t\}$,
\item $\rho$ is Markov along $s$ for every $s \in S\setminus \{t\}$.
\end{itemize}

\end{lem}

%For the sake of notation, let $\La=\langle t\rangle\le \F $ and let $\G\le \F$ be the subgroup generated by $S\setminus \{t\}$ so that $\F$ decomposes as the free product $\F = \La * \G$. 

\begin{proof}

 By Lemma \ref{lem:key} there exists an orbit equivalence 
$$\Psi:A^\Z \to A^\Z$$
from the shift action  $\Z\cc^T (A^\Z, \mu_{t})$ to $\Z\cc^T (A^\Z, \nu)$ such that $\pi = \pi \Psi.$

Define
\begin{eqnarray*}
%T:A^\Z \to A^\Z &\textrm{ by } & (Tx)_n = x_{n+1} \\
\tT:A^\Z \to A^\Z &\textrm{ by } & \tT = \Psi^{-1} T \Psi \\
\a: \Z \times A^\Z \to \Z &\textrm{ by } & \tT^{\a(n,x)} x = T^n x\\
\b: \Z \times A^\Z \to \Z &\textrm{ by } & T^{\b(n,x)} x = \tT^n x
\end{eqnarray*}
 Because $\Psi$ is an OE, $\a,\b$ are well-defined, satisfy the cocycle equations below and the inverse equation:
\begin{eqnarray}
\a(n+m,x) &=& \a(n, T^m x) + \a(m,x) \label{eqn:b1} \\
\b(n+m,x) &=& \b(n, \tT^m x) + \b(m,x) \label{eqn:b2} \\
\b(\a(n,x),x) &=& \a(\b(n,x),x) = n \label{eqn:b3}
\end{eqnarray}
for all $x\in A^\Z, n,m  \in \Z$. To ease notation, let $R=R^t: A^\F \to A^\Z$ denote the restriction map as in Definition \ref{defn:r}. Define $\tau:(S\cup S^{-1}) \times A^\F \to \F$ by
$$\tau(s,x) = s \textrm{ for } s\in S \cup S^{-1}\setminus \{t,t^{-1}\},$$
$$\tau(t^n,x) = t^{\b(n,Rx)} \textrm{ for } n \in \{-1,+1\}.$$
Then
\begin{eqnarray*}
\tau(t^{-1}, \tau(t,x)\cdot x) &=& \tau(t^{-1}, t^{\b(1,Rx)} \cdot x) = t^{\b(-1, R( t^{\b(1,Rx)} \cdot x) )} \\
&=& t^{\b(-1, T ^{\b(1,Rx)}Rx )} =  t^{\b(-1, \tT Rx) }  = t^{- \b(1,Rx)} = \tau(t,x)^{-1}
\end{eqnarray*}
where the second-to-last equality follows from the cocycle equation for $\b$. Similarly, $\tau(t, \tau(t^{-1},x)\cdot x) = \tau(t^{-1},x)^{-1}$. So 
 $\tau$ satisfies the conditions of Lemma \ref{lem:basic1}.  Let $\ast, \omega, \Omega$ be as in Lemma \ref{lem:basic1}. 

We claim that
\begin{eqnarray}\label{eqn:o4}
t^n \ast x = t^{\b(n,Rx)} \cdot x  \textrm{ for } n \in \Z.
\end{eqnarray}
By definition this statement is true if $n\in \{-1,1\}$. So it suffices to prove that if $n,m \in \Z$ satisfy (\ref{eqn:o4}) then $n+m$ also satisfies (\ref{eqn:o4}). We claim that $R(t^m\ast x) = \tT^m Rx$. This follows from:
\begin{eqnarray*}
R(t^m\ast x)_n &=& (t^m \ast  x)_{t^n} = (t^{\b(m,Rx)}\cdot x)_{t^n} = x_{t^{n + \b(m,Rx)}} \\
&=& (Rx)_{n + \b(m,Rx)} = (T^{\b(m,Rx)}Rx)_n = (\tT^mRx)_n.
\end{eqnarray*}
Therefore, 
\begin{eqnarray*}
t^{n+m}\ast x &=& t^n\ast (t^m\ast  x) = t^{\b(n, R t^m\ast x)} \cdot (t^m \ast x) = t^{\b(n, Rt^m\ast x)} \cdot (t^{\b(m,Rx)} \cdot x) \\
&=& t^{\b(n, Rt^m\ast x) + \b(m,Rx)} \cdot x =   t^{\b(n, \tT^mR x) + \b(m,Rx)} \cdot x =  t^{\b(n+m, R x)} \cdot x.
\end{eqnarray*}
This proves (\ref{eqn:o4}).

To finish the lemma, we will show that $\ast$ has the same orbits as $\cdot$ (modulo $\mu$ null sets) and afterwards that $\F \cc^{\ast} (A^\F,\mu)$ is measure-conjugate to $\F \cc(A^\F,\rho)$.

\noindent {\bf Claim 1}. $\F \cc^\ast (A^\F,\mu)$ has the same orbits  as $\F \cc(A^\F,\mu)$.

\begin{proof}

Let $\k:\F \times A^\F \to \F$ be the unique function satisfying the following:
\begin{eqnarray*}
\k(s,x) & = & s \textrm{ for } s\in S\cup S^{-1} -\{t, t^{-1}\} \\
\k(t^n,x) &=& t^{\a(n,Rx)} \textrm{ for } n\in \Z \\
\k(gh,x) &=& \k(g, h \cdot x) \k(h,x).
\end{eqnarray*}
This is well-defined because $\a$ satisfies the cocycle equation and $R(t^n\cdot x)=T^nRx$. The following inverse equation also holds: 
\begin{eqnarray}\label{eqn:inverse}
\k(\omega(h,x),x) = \omega(\k(h,x),x) = h.
\end{eqnarray}
To see this, first note that it is obvious when $h \in S \cup S^{-1} \setminus \{t,t^{-1}\}$. Then observe that the inverse equations (\ref{eqn:b3}) imply the statement for $h$ in the subgroup generated by $t$.  If the statement holds for elements $h_1,h_2$ then it must hold for their product because of:
\begin{eqnarray*}
\k(\omega(h_1h_2,x),x) &=&\k(\omega(h_1, h_2 \ast x) \omega(h_2, x),x) = \k(\omega(h_1, h_2\ast x), \omega(h_2,x) \cdot x) \k(\omega(h_2,x),x) \\
&=& \k(\omega(h_1, h_2\ast x), h_2 \ast x) \k(\omega(h_2,x),x) = h_1h_2
\end{eqnarray*}
and the related equation with the orders of $\omega,\k$ reversed. By induction (\ref{eqn:inverse}) is true for all $h \in \F$. Equation (\ref{eqn:inverse}) implies that the actions $\ast$ and $\cdot$ have the same orbits (modulo $\mu$ null sets). 
\end{proof}

\noindent {\bf Claim 2}. $\Omega$ is a measure-conjugacy  between $\F \cc^\ast (A^\F,\mu)$ and $\F \cc(A^\F,\Omega_*\mu)$. 

\begin{proof}
By Lemma \ref{lem:basic1}, $\Omega$ is $(\ast,\cdot)$-equivariant.  In order to show that $\Omega$ is invertible, define
\begin{eqnarray*}
\hT:A^\Z \to A^\Z &\textrm{ by } &  \hT = \Psi T \Psi^{-1} \\
\hat{\b}: \Z \times A^\Z \to \Z &\textrm{ by } & T^{\hat{\b}(n,x)} x = \hT^n x.
\end{eqnarray*}
Because $\Psi$ is an OE, $\hat{\b}$ satisfies the cocycle equation
\begin{eqnarray}\label{eqn:b4}
\hat{\b}(n+m,x) = \hat{\b}(n, \hT^m x) + \hat{\b}(m,x).
\end{eqnarray}
It also satisfies the inverse equations:
\begin{eqnarray}\label{eqn:b5}
\hat{\b}(\b(n, \Psi^{-1}x),x) & = & \b(\hat{\b}(n, \Psi x), x) = n.
\end{eqnarray}
To see this observe that
\begin{eqnarray*}
T^{\hat{\b}(\b(n, \Psi^{-1}x),x)}x & = &  \hT^{\b(n,\Psi^{-1}x)} x = \Psi T^{\b(n,\Psi^{-1}x)} \Psi^{-1}x \\
&=& \Psi \tT^n \Psi^{-1} x = T^nx.
\end{eqnarray*}
This shows $\hat{\b}(\b(n, \Psi^{-1}x),x) = n$. The other equality is similar.

Define $\hat{\tau}:(S\cup S^{-1})\times A^\F \to \F$ by
$$\hat{\tau}(s,x) = s  \textrm{ for } s\in S \cup S^{-1}\setminus \{t,t^{-1}\},$$
$$\hat{\tau}(t^n,x) = t^{\hat{\b}(n,Rx)}   \textrm{ for } n \in \{-1,1\}.$$
As in the case of $\tau$, $\hat{\tau}$ satisfies the hypotheses of Lemma \ref{lem:basic1}. Let $\star,\homega,\hat{\Omega}$ be the action, cocycle and map defined by Lemma \ref{lem:basic1}. 

The following restriction equations hold:
\begin{eqnarray}
R \hat{\Omega}  = \Psi^{-1}R, \quad  R \Omega  = \Psi R.  \label{R1}
\end{eqnarray}
The first equation above is proven by:
\begin{eqnarray*}
(R \hat{\Omega} x)_n &=& (\hat{\Omega}x)_{t^n} = x_{\homega(t^n,x)} = x_{t^{\hat{\b}(n,Rx)}} \\
&=& (T^{\hat{\b}(n,Rx)}Rx)_0 =  ( \hT^n Rx)_0 = ( \Psi T^n \Psi^{-1} Rx)_0 = \pi(  \Psi T^n \Psi^{-1} Rx) \\
&=& \pi( T^n \Psi^{-1} Rx) = ( \Psi^{-1}Rx)_n.
\end{eqnarray*}
The second equation is similar. We now claim the following inverse equations:
\begin{eqnarray}\label{injective}
\homega(\omega(g,\hat{\Omega}x),x) = \omega(\homega(g,\Omega x),x) = g.
\end{eqnarray}
This is immediate if $g \in S\cup S^{-1} \setminus \{t,t^{-1}\}$. The case $g=t^n$ follows from (\ref{R1}) and (\ref{eqn:b5}):
\begin{eqnarray*}
\homega(\omega(t^n,\hat{\Omega}x),x) = \homega(t^{\b(n,R\hat{\Omega} x)},x) = \homega(t^{\b(n,\Psi^{-1} R x)},x) = t^{\hat{\b}(\b(n,\Psi^{-1}Rx),Rx)} = t^n.
\end{eqnarray*}
The general case follows from Lemma \ref{lem:basic3} which also shows $\hat{\Omega}\Omega x= \Omega \hat{\Omega}x = x$. Therefore, $\Omega$ is invertible with inverse equal to $\hat{\Omega}$.  
\end{proof}

The measure $\rho$ is well-defined by Lemma \ref{lem:helper0}. It now suffices to show $\Omega_*\mu=\rho$. This is obtained by verifying that $\Omega_*\mu$ satisfies the same conditions defining $\rho$.

\noindent {\bf Claim 3}. For every $s\in S$, $(\Omega_*\mu)_s = \rho_s.$

\begin{proof}
By (\ref{R1}),
$$(\Omega_*\mu)_t = (R\Omega)_*\mu = (\Psi R)_*\mu = \Psi_* \mu_t = \nu=\rho_t.$$

Fix $s\in S\setminus \{t\}$. Then  $\omega(s^n,x)=s^n$ for all $n$. We claim that $R^s\Omega = R^s$. This follows from
$$(R^s\Omega x)_n = (\Omega x)_{s^n} = x_{\omega(s^n,x)} = x_{s^n} = (R^sx)_n.$$
So
$$(\Omega_*\mu)_s = (R^s \Omega)_*\mu = R^s_*\mu = \mu_s=\rho_s.$$
\end{proof}

%Claim 3 and Lemma \ref{lem:helper0} imply $\Omega_*\mu=\rho$. 

It now suffices to show that $\Omega_*\mu$ is $s$-Markov for every $s\in S\setminus \{t\}$. We will use Lemma \ref{lem:basic2} and the next two claims.

\noindent {\bf Claim 4}. For any $s\in S  \cup S^{-1} \setminus \{t,t^{-1}\}$, 
$$g \in \past(s) \Rightarrow \omega(g,x) \in \past(s) ~\forall x\in A^\F.$$

\begin{proof}
%It follows immediately from the definition of past that if $h \in S \cup S^{-1}$ and $k\in \past(s)$ then either $hk \past(s)$ or ($h=s^{-1}$ and $k=s$).  

The proof of the Claim is by induction on $|g|$. If $|g|\le 1$ then $g=s$ and $\omega(g,x)=s \in \past(s)$. So assume $|g|>1$. Then $g=hk$ for some $h \in S \cup S^{-1}$ and $k\in \past(s)$ with $|k|< |g|$. So
$$\omega(g,x)=\omega(hk,x) = \omega(h, k \ast x)\omega(k,x).$$
By induction, we may assume $\omega(k,x) \in \past(s)$. To obtain a contradiction, suppose $\omega(g,x) \notin \past(s)$.

 If $h \in S \cup S^{-1} \setminus \{t,t^{-1}\}$ then $\omega(h, k \ast x) = h$ has length 1. If $f \in \F$ is any element with length $1$ then $f\past(s) \subset \{e\} \cup \past(s)$. So if  $\omega(g,x)= \omega(h, k \ast x)\omega(k,x) \notin \past(s)$ then $\omega(g,x)=e$. But $\omega$ is injective by (\ref{injective}). This implies $g=e$ contradicting that $g \in \past(s)$.  
 
 On the other hand, if $h=t^m$ for some $m$ then $\omega(h,k\ast x) = t^n$ for some $n$. Since $t^n \past (s) \subset \past(s)$ (since $s\notin \{t,t^{-1}\}$), this shows $\omega(g,x) \in \past(s)$.
\end{proof}

\noindent {\bf Claim 5}. If $g \in \past(t) \cup \past(t^{-1})$ then $\omega(g,x) \in \past(t) \cup \past(t^{-1})$   for $\mu$-a.e. $x\in A^\F.$

\begin{proof}
 If $|g|\le 1$ then $g \in \{t,t^{-1}\}$ and 
$$\omega(g,x)=t^n \in \past(t) \cup \past(t^{-1})$$
for some $n$ by (\ref{eqn:o4}) and the definition of $\omega$ from Lemma \ref{lem:basic1}. We are using here that $\omega$ is injective by (\ref{injective}) and therefore $n\ne 0$. 

So we may assume $|g|>1$. Then $g=hk$ for some $h \in S \cup S^{-1}$ and $k \in \past(t)\cup \past(t^{-1})$ with $|k|<|g|$. So
$$\omega(g,x)=\omega(hk,x) = \omega(h, k \ast x)\omega(k,x).$$
By induction we may assume $\omega(k,x) \in \past(t) \cup \past(t^{-1})$. To obtain a contradiction, assume $\omega(g,x) \notin \past(t) \cup\past(t^{-1})$. 

If $h \in S \cup S^{-1} \setminus \{t,t^{-1}\}$ then $\omega(h,k\ast x)=h$. Since
$$(S \cup S^{-1} \setminus \{t,t^{-1}\})[\past(t) \cup \past(t^{-1})] \subset \past(t) \cup \past(t^{-1}),$$
 this shows $\omega(g,x) \in \past(t)\cup \past(t^{-1})$. 

So assume $h = t^n$ for some $n \in \{-1,+1\}$. Then $\omega(h,k\ast x)=t^m$ for some $m$. Since 
$$t^m [ \past(t) \cup \past(t^{-1}) ] \subset  \past(t) \cup \past(t^{-1}) \cup \{e\}$$
it follows that $\omega(k, x) = t^{-m}$ so $\omega(g,x)=e$. But this implies $g=e$ since $\omega$ is injective (\ref{injective}), a contradiction. 
\end{proof}

Lemma \ref{lem:basic2} implies $\Omega_*\mu$ is $s$-Markov for every $s\in S\setminus \{t\}$. So $\Omega_*\mu$ satisfies the same defining properties as $\rho$. Thus $\Omega_*\mu=\rho$.

\end{proof}

\begin{proof}[Proof of Theorem \ref{thm:main}]
Let $\mu \in \Prob_\F(A^\F)$ be Markov and properly ergodic. Since all Bernoulli shifts over $\F$ are OE (by \cite{MR2763777}) it suffices to show that $\F\cc (A^\F,\mu)$ is OE to a Bernoulli shift. By Proposition \ref{prop:reduction} we may assume that $\mu_s$ is essentially free and ergodic for every $s\in S$. 

Let $S=\{s_1,\ldots,s_r\}$. Let $\nu$ be the Bernoulli product measure $\nu = (\pi_*\mu)^\Z \in \Prob_\Z(A^\Z)$. Define shift-invariant measures $\mu^{(0)},\mu^{(1)},\ldots, \mu^{(r)} \in \Prob_\F(A^\F)$ as follows. First, $\mu^{(0)}=\mu$. For $i>0$, $\mu^{(i)}$ is characterized by:
\begin{itemize}
\item $\mu^{(i)}_{s_j} = \nu$ for all $j\le i$
\item $\mu^{(i)}_{s_j} = \mu_{s_k}$ for all $i<j\le r$,
\item $\mu^{(i)}$ is $s_j$-Markov for all $j \ne i$.
\end{itemize}
By Lemma \ref{lem:helper0}, $\mu^{(i)}$ is Markov. By Lemma \ref{lem:main},  $\F \cc (A^\F,\mu^{(i)})$ is OE to $\F \cc (A^\F,\mu^{(i+1)})$ for all $i<r$. Since $\mu^{(r)}=(\pi_*\mu)^\F$ is Bernoulli this completes the proof.
\end{proof}

%\cite{ElekSzabo2004}

\bibliography{biblio}
\bibliographystyle{alpha}

\end{document}